\documentclass{amsart}

\usepackage{amssymb, amsmath, amsthm, pictexwd}

\newcommand {\C} [1] {{\mathbb C}^{#1}}

\newcommand {\R} [1] {{\mathbb R}^{#1}}

\newcommand {\pdd} [1] {\displaystyle{\frac{\partial }{\partial
{#1}}}}

\newcommand {\real} {\mbox{Re}}

\newcommand {\imag} {\mbox{Im}}

\newtheorem{theorem}{Theorem}[section]

\newtheorem{lemma}[theorem]{Lemma}

\newtheorem{cor}[theorem]{Corollary}

\newtheorem{prop}[theorem]{Proposition}

\newtheorem{definition}[theorem]{Definition}

\newtheorem{remark}[theorem]{Remark}

\numberwithin{equation}{section}

\begin{document}
\title{The Szeg\"o Kernel for Non-Pseudoconvex Tube Domains in $\C{2}$}
\author{Michael Gilliam and Jennifer Halfpap}
\thanks{The second author was partially supported by NSF grant DMS 0654195.}
\maketitle

\begin{abstract}
We consider the Szeg\"o kernel associated with domains $\Omega$ in
$\C{2}$ given by
$$\Omega=\{\,(z,w):\imag \,w > b(\real \,z)\,\} $$
for $b$ a {\it non-convex} polynomial of even degree with positive
leading coefficient. Such domains are not pseudoconvex.  We give a
precise description of a subset of $\overline{\Omega}\times
\overline{\Omega}$ on which the kernel and all of its derivatives
are finite.  We show, in particular, that for such domains, the
Szeg\"o kernel has singularities off the diagonal of $\partial
\Omega \times
\partial \Omega$ as well as points on the diagonal at which it is
finite.
\end{abstract}



\section{Introduction}

Let $\Omega \subset \C{n}$ be a domain with smooth boundary
$\partial \Omega$, and let $\mathcal{O}(\Omega)$ denote the space
of holomorphic functions on $\Omega$.  Associated with such a
domain are certain operators: the Bergman projection $\mathcal{B}$
and the Szeg\"o projection $\mathcal{S}$. The former is the
orthogonal projection of $L^2(\Omega)$ onto the closed subspace
$L^2(\Omega) \cap \mathcal{O}(\Omega)$, whereas the latter is the
orthogonal projection of $L^2(\partial \Omega)$ onto the closed
subspace $\mathcal{H}^2(\Omega)$ of boundary values of elements of
$\mathcal{O}(\Omega)$. An important goal of research on these
operators is to obtain results concerning their mapping properties
(e.g., conditions under which they extend to bounded operators on
the appropriate $L^p$ spaces).

Often, understanding these operators begins with an investigation
of the associated integral kernel; one identifies distributions
$B$ and $S$ such that for $f \in L^2 (\Omega)$ and $g\in
L^2(\partial \Omega)$,
$$\mathcal{B}[f](z)=\int_{\Omega} f(w) B(z,w) \,dw$$
$$\mathcal{S}[g](z)=\int_{\partial \Omega}g(w) S(z,w) \,d\sigma(w).$$
A first step in the analysis of these kernels is to describe the
subset of $\partial \Omega \times \partial \Omega$ to which they
(and their derivatives) extend continuously. An early result of
this sort is due to Kerzman \cite{Kerzman:72}, who uses Kohn's
formula to show that for $\Omega \subseteq \C{n}$ bounded and
strongly pseudoconvex, $B$ and its derivatives extend continuously
to $(\overline{\Omega}\times \overline{\Omega})\setminus \Delta$,
where $\Delta=\{\,(z,w)\in
\partial \Omega \times\partial \Omega : z=w\,\}$ is the diagonal of the boundary.

A further step in the analysis is to obtain sharp size estimates
for these kernels and their derivatives near their singularities,
together with mapping properties of the associated operators. This
is done, for instance, by Nagel, Rosay, Stein, and Wainger
\cite{NRSW:89} for finite-type domains in $\C{2}$ and by McNeal
and Stein (\cite{McNeal:94},
\cite{McNealStein:94},\cite{McNealStein:97}) for convex domains in
$\C{n}$.

In contrast with the situation for pseudoconvex domains,
comparatively little is known about the Szeg\"o kernel for
non-pseudoconvex domains, even in $\C{2}$. Consider
\begin{equation}\label{tube domain}
\Omega =\{\,(z_1=x+iy,z_2=t+i\xi): \xi > b(x)  \,\},
\end{equation}
for a real-valued smooth function $b$ satisfying $\lim_{|x| \to
\infty}b(x)/|x|=\infty$. This domain is pseudoconvex precisely
when $b$ is convex. Some of the first results concerning the
Szeg\"o kernel in the non-pseudoconvex context are due to
Carracino (\cite{CarracinoPHD}, \cite{Carracino:07}). She obtains
detailed estimates for the Szeg\"o kernel on the boundary of a
model domain of the type \eqref{tube domain} with $b$ a
non-smooth, non-convex, piecewise quadratic function.  She shows
that the Szeg\"o kernel has singularities off of $\Delta$ in this
case. Then in \cite{GH:11}, the current authors identify a subset
of $\overline{\Omega} \times \overline{\Omega}$ on which the
integrals defining the Szeg\"o kernel and its derivatives are
absolutely convergent for the case in which $b$ is a non-convex
quartic polynomial.  In particular, this work shows that there are
points {\it on the diagonal} $\Delta$ at which the Szeg\"o kernel
is finite as well as points {\it off the diagonal} at which it is
infinite.

In this paper, we explore this phenomenon in the much more general
setting in which $b$ is a non-convex even-degree polynomial with
positive leading coefficient.  Without loss of generality, we may
suppose
\begin{equation}\label{our b}
b(x)=\frac{1}{2n}x^{2n}+\sum_{j=2}^{2n-1} a_j x^j, \quad n \ge 2.
\end{equation}
Although the statements of the theorems in this paper closely
resemble those in \cite{GH:11}, the technical challenges in
proving the theorems are rather different.  We will comment on
these substantial differences in due course.

We close this introductory section with a note on the motivations
for studying {\it non-pseudoconvex} domains. To begin with, we are
motivated by an interest in singular integral operators.  The
Szeg\"o kernel for a pseudoconvex domain of finite type is an
example of a non-isotropic smoothing operator; this is a class of
operators that is well-understood.  (See \cite{NRSW:89}).
Carracino's work shows that the structure of the singularities of
the Szeg\"o kernel can be very different in the non-pseudoconvex
setting, but it is inconclusive on the question of whether these
kernels are related to flag kernels \cite{NRS:01} or product
singular integral operators \cite{NS:04}.

A second motivation for exploring the non-pseudoconvex setting
arises because of its possible connection to the perhaps more
natural problem of understanding the Szeg\"o projection operator
associated with a CR manifold of higher codimension such as the
tube model
$$T_4=\{\,(z,w_2,w_3,w_4):\real \, w_j= [\real \,z]^j\,\}.$$
One can derive an expression analogous to \eqref{Szego kernel on
R3 X R3} for the kernel associated with the orthogonal projection
of $L^2(T_4)$ onto the subspace of functions annihilated by the
tangential Cauchy-Riemann operators.  The resulting integral is
even more complicated in that setting; nonetheless, many of the
technical challenges arising in that setting arise in the
non-pseudoconvex setting as well.  The analysis in the
non-pseudoconvex setting can thus guide some of the analysis for
higher-codimensional CR manifolds.

\section{Definitions, Notation, and Statement of Results}

We begin with a more precise discussion of the Szeg\"o projection
operator and its associated integral kernel for domains in $\C{2}$
having the form \eqref{tube domain}. We take $b$ smooth so that
$\Omega \subset \C{2}$ is smoothly-bounded. As above, let
$\mathcal{O}(\Omega)$ denote the space of functions holomorphic on
$\Omega$.  Define
$$\mathcal{H}^2(\Omega):=\left\{\,F\in \mathcal{O}(\Omega): \sup_{\varepsilon > 0} \int_{\partial \Omega} |F(x+iy, t+ib(x)+i\varepsilon)|^2\, dx\,dy\,dt<\infty\,\right\}.$$
$\mathcal{H}^2(\Omega)$ can be identified with the set of all
functions $f$ in $L^2(\partial \Omega)$ (which is itself
identified with $L^2(\R{3})$) which are solutions in the sense of
distributions to
\begin{equation}\label{CR equation}
\left(\pdd{x} + i\pdd{y} -ib'(x) \pdd{t} \right)[f] \equiv 0.
\end{equation}
We define the {\it Szeg\"o projection operator $\mathcal{S}$} to
be the orthogonal projection of $L^2(\partial \Omega)$ onto this
(closed) subspace $\mathcal{H}^2(\Omega)$.

One establishes the existence of a unique integral kernel
associated with the operator.  This is discussed, for example, in
\cite{St(BB):72}, where the approach is as follows: Begin with an
orthonormal basis $\{\phi_j\}$ for $\mathcal{H}^2(\Omega)$ and
form the sum
$$S(z,w)=\sum_{j=1}^{\infty} \phi_j(z)\overline{\phi_j(w)}.$$
One shows that this converges uniformly on compact subsets of
$\Omega \times \Omega$, that $\overline{S(z , \cdot)}\in
\mathcal{H}^2(\Omega)$ for each $z \in \Omega$, and that for $g
\in \mathcal{H}^2(\Omega)$,
$$g(z)=\int_{\partial \Omega} S(z,w) g(w)\,d\sigma(w).$$
$S$ is then the {\it Szeg\"o kernel}.  From its construction it is
clear that it will be smooth on $\Omega \times \Omega$.  It may
extend to a smooth function on some larger subset of
$\overline{\Omega} \times \overline{\Omega}$.

For domains of the form \eqref{tube domain}, one can derive an
explicit formula for the Szeg\"o kernel.  Let $z=(z_1, z_2)$ and
$w=(w_1, w_2)$ be elements of $\C{2}$.  Set
\begin{equation}\label{N(eta,tau)}
N(\eta, \tau)=\int_{-\infty}^\infty e^{2\tau[\eta \lambda -
b(\lambda)]}\,d\lambda.
\end{equation}
Then
\begin{equation}\label{Szego kernal inside}
S(z,w)=c\int \!\!\!\int_{\tau > 0} \tau e^{\eta \tau [z_1 +
\bar{w}_1]+i\tau[z_2 - \bar{w}_2]} [N(\eta, \tau)]^{-1} \,d\eta
\,d\tau,
\end{equation}
where $c$ is an absolute constant.
\begin{remark}
See \cite{HNW:09} for detailed discussions of $\mathcal{H}^p$
spaces for unbounded domains, the derivations of such integral
formulas, and the identification of $\mathcal{H}^2(\Omega)$ with
$L^2(\partial \Omega)(=L^2(\R{3}))$ functions satisfying the
differential equation \eqref{CR equation}.
\end{remark}
\begin{remark}
Many authors only consider $S$ as a distribution on $\partial
\Omega \times
\partial \Omega$ since $S$ is smooth on $\Omega \times \Omega$.  In this situation, one can identify the boundary with $\R{3}$ and
consider the integral kernel
\begin{eqnarray}\label{Szego kernel on R3 X R3}
\lefteqn{S[(x,y,t),(r,s,u)]=}\nonumber\\
& &c\int_0^\infty \!\!\!\int_{-\infty}^\infty \tau
e^{\tau[i(t-u)+i \eta (y-s)-[b(x)+b(r)-\eta(x+r)]]}\left[N(\eta,
\tau)\right]^{-1} \,d\eta \,d\tau.\label{szego kernel for tube}
\end{eqnarray}
This is done, for example, in the work of Nagel
\cite{Nag(Beijing):86}, Haslinger \cite{Haslinger:95}, and
Carracino \cite{CarracinoPHD}, \cite{Carracino:07}.
\end{remark}

We may now state our results:

Let $b$ be as in \eqref{our b}. For each real $\eta$, set
$B_\eta(x):=-\eta x + b(x)$. The set of minimizers of this
function is of vital importance in our analysis.  Thus we define
\begin{equation}\label{Lambda}
\Lambda_\eta=\{\,\lambda:\inf_x B_\eta(x)=B_\eta(\lambda)\,\},
\end{equation}
with $\Lambda := \bigcup_\eta \Lambda_\eta$ and
$\mathcal{C}=\{\,\eta:|\Lambda_\eta|>1\,\}$.  Furthermore, set
\begin{gather}
z=(z_1,z_2)=(x+iy, t+ib(x)+i h)\\
w=(w_1,w_2)=(r+is,u+ib(r)+i k),
\end{gather}
and define
\begin{equation}\label{sigma}
\Sigma=\{\,(z,w): x=r \;\text{and}\; x\in\Lambda\,\}\cup
\{\,(z,w): x,r \in \Lambda_c \;\text{for}\; c\in\mathcal{C} \,\}.
\end{equation}
Finally, for a function $b$ continuous on $\mathbb{R}$, define the
{\it Legendre transform} of $b$ by
\begin{equation}
b^*(\eta):=\sup_{x \in \mathbb{R}} [\eta x - b(x)]=-\inf_x
B_\eta(x).
\end{equation}
\begin{theorem}\label{theorem: absolute convergence}
The integral defining $S(z,w)$ is absolutely convergent in the
region in which
\begin{equation}\label{region for abs conv}
h+k+b(x)+b(r)-2b^{**}\left(\frac{x+r}{2}\right)>0.
\end{equation}
This is an open neighborhood of $(\overline{\Omega}\times
\overline{\Omega})\setminus \Sigma$. More generally, if $i_1$,
$j_1$, $i_2$, and $j_2$ are non-negative integers, then
\begin{equation}\label{derivatives of S}
\partial^{i_1}_{z_1} \partial^{j_1}_{\bar{w_1}} \partial^{i_2}_{z_2}
\partial^{j_2}_{\bar{w_2}}S(z,w)=c' \int \!\!\! \int_{\tau>0} e^{\eta \tau
[z_1+\bar{w_1}]+i\tau[z_2-\bar{w_2}]}\frac{\eta^{i_1+j_1}
\tau^{i_1+j_1+i_2+j_2+1}}{N(\eta,\tau)}\,d\eta\,d\tau
\end{equation}
is absolutely convergent in the same region.
\end{theorem}
\begin{remark}
Compare this with Theorem 3.2 in \cite{HNW:09} and with Theorem
2.3 in \cite{GH:11}.
\end{remark}
\begin{theorem}\label{theorem: x=-r singularity}
If $[(x+iy,t+ib(x)),(r+iy,t+ib(r))] \in \Sigma$,
$\mathcal{S}[(x,y,t),(r,y,t)]$ is infinite. Also, if $\delta =
h+k> 0$,
$$\lim_{\delta \to 0^+} S[(x+iy,t+i(b(x)+h)),(r+iy,t+i(b(r)+k))]=\infty.$$
\end{theorem}
We will show that the set $\Sigma$ is equal to the diagonal
$\Delta$ of $\partial \Omega \times \partial \Omega$ precisely
when the polynomial $b$ is convex. For non-convex $b$, there are
both points {\it off the diagonal} that are contained in $\Sigma$
and points {\it on the diagonal} that are not in $\Sigma$. We
summarize this important observation in a corollary.
\begin{cor}
For tube domains \eqref{tube domain} in $\C{2}$ with $b$ an
even-degree polynomial with positive leading coefficient, the
Szeg\"o kernel extends smoothly to $(\overline{\Omega} \times
\overline{\Omega}) \setminus \Delta$ if and only if $b$ is convex.
\end{cor}

An analysis of the Szeg\"o kernel begins with estimates of the
integral $N$ defined in \eqref{N(eta,tau)}. Observe that for fixed
$\eta \in \mathbb{R}$ and $\tau >0$,
$\lim_{|\lambda|\to\infty}2\tau[\eta \lambda - b(\lambda)]=-2\tau
\lim_{|\lambda|\to \infty}B_{\eta}(\lambda)=-\infty$ . The
heuristic principle that guides the analysis of such integrals is
that the main contribution comes from a neighborhood of the
point(s) at which the exponent attains its global maximum. For our
integral $N$, let $\lambda(\eta)$ denote the largest real number
at which $\inf_{\lambda}B_\eta(\lambda)$ is attained. Then
\begin{eqnarray*}
N(\eta,\tau)&=&e^{-2\tau B_\eta(\lambda(\eta))}
\int_{-\infty}^\infty e^{-2\tau [-\eta \lambda +
b(\lambda)-B_\eta(\lambda(\eta))]}\,d\lambda\\
&=&e^{2\tau b^*(\eta)}\int_{-\infty}^\infty e^{-2\tau
p_\eta(\xi)}\,d\xi,
\end{eqnarray*}
where
\begin{equation}\label{def of poly p_eta}
p_\eta(\xi):=-\eta \xi + b(\xi + \lambda(\eta))-b(\lambda(\eta))
\end{equation}
is a non-negative polynomial vanishing to even order at the
origin. Furthermore, by our choice of $\lambda(\eta)$, if for some
$\eta$, $p_\eta(\xi)=0$ for non-zero $\xi$, necessarily $\xi < 0$.

In Sections \ref{section: global properties} and \ref{section:
local properties of B}, we focus on understanding the main
contribution to the integral $N$ by exploring $B_\eta$ and
$\lambda(\eta)$, while in Section \ref{section: estimates for I},
we focus on estimating the integral that remains once we have
taken out this main contribution. The theorems are established in
Section \ref{section: proof of thms}.

\section{\label{section: global properties} Global properties of $\lambda(\eta)$ and $B_\eta$}

This section contains a number of technical lemmas on the
long-term behavior of  $\lambda(\eta)$ and
$B_\eta(\lambda(\eta))=-b^*(\eta)$. Most of these results follow
rather easily from the fact that $B_\eta$ is a polynomial and
$\lambda(\eta)$ is one of its critical points.

\begin{lemma}\label{lemma: asy behavior of lambda}
$\lim_{\eta \to -\infty}\lambda(\eta)=-\infty$ and $\lim_{\eta \to
\infty}\lambda(\eta)=\infty$. Furthermore, $\lambda(\eta) \sim
\eta^{\frac{1}{2n-1}}$ as $|n|\to\infty$.
\end{lemma}
\begin{proof}
We consider the case $\eta \to -\infty$.  The case $\eta \to
\infty$ is established similarly.

Consider the equation $b'(\omega)=\eta$.  Since $b$ has even
degree and positive leading coefficient, there exists an interval
$(-\infty, \beta)$ on which $b$ is convex.  Thus on this interval,
$b'$ is an increasing function with a well-defined inverse
function $\eta \mapsto \omega(\eta)$. We claim that for any $L>0$
with $-L \le \beta$ there exists $m$ such that for $\eta < m$,
$\omega(\eta)=\lambda(\eta)$. Indeed, since $b'$ is an odd-degree
polynomial with positive leading coefficient, the number
$m=\inf\{b'(\omega):\omega \ge -L \}$ is finite. If $\eta < m$,
the only solution to $b'(\omega)=\eta$ on $\mathbb{R}$ must lie in
$(-\infty, -L)\subseteq (-\infty, \beta)$.  Since $\lambda(\eta)$
is a solution, it lies in this interval. Thus for $\eta<m$,
$\lambda(\eta)=\omega(\eta)$.

Note that $b'(\omega)=\omega^{2n-1}+\sum_{j=2}^{2n-1}j a_j
\omega^{j-1}$. Take $L>0$ so that $|\omega|\ge L$ implies
$$\sum_{j=2}^{2n-1}j|a_j||\omega|^{-2n+j}\le\frac{1}{2}.$$  Then
for $\omega \le -L$,
$$\frac{3}{2}\omega^{2n-1} \le \omega^{2n-1}\left(1+\sum_{j=2}^{2n-1}j|a_j||\omega|^{-2n+j}\right)=\omega^{2n-1}-\sum_{j=2}^{2n-1}j|a_j||\omega|^{j-1} \le b'(\omega). $$
Since for $\eta < m$ the solution to $b'(\omega)=\eta$ is
$\lambda(\eta)$, this shows that $\lambda(\eta) \to -\infty$ as
$\eta \to -\infty$. Furthermore, if $b'(\omega)=\eta$,
\begin{equation}\label{asy fomula for lambda}
\omega^{2n-1}=\eta-\sum_{j=2}^{2n-1} j a_j\omega^{j-1}
\Leftrightarrow
1=\frac{\eta}{\omega^{2n-1}}-\sum_{j=2}^{2n-1}\frac{j
a_j}{\omega^{2n-j}}=\frac{\eta}{\omega^{2n-1}}+o(1)
\end{equation}
as $\eta \to -\infty$.  Thus $\lambda(\eta)^{2n-1} \sim \eta$ as
$\eta \to -\infty$, i.e., $\lambda(\eta)^{2n-1} =\eta[1+o(1)]$ as
$\eta \to -\infty$.  It follows that $\lambda(\eta) \sim
\eta^{\frac{1}{2n-1}}$ as $\eta \to -\infty$.
\end{proof}

This allows us immediately to obtain size estimates for
$B_\eta(\lambda(\eta))=-b^*(\eta)$ for large $\eta$.
\begin{lemma}\label{lemma: asy behavior of b*}
\begin{equation}\label{asy for b*}
b^*(\eta)\sim
\left(\frac{2n-1}{2n}\right)\eta^{\frac{2n}{2n-1}}\quad\text{as
$|\eta|\rightarrow\infty.$}
\end{equation}
\end{lemma}
\begin{proof}
\begin{eqnarray*}
B_\eta(\lambda(\eta))&=&b(\lambda(\eta))-\eta \lambda(\eta)\\
&=&\frac{1}{2n}\lambda(\eta)^{2n}
+\sum_{j=2}^{2n-1}a_j\lambda(\eta)^j-\eta \lambda(\eta).
\end{eqnarray*}
By Lemma \ref{lemma: asy behavior of lambda},
\begin{eqnarray*}\label{asymptotic2}
B_\eta(\lambda(\eta))&=&\frac{1}{2n}\eta^{\frac{2n}{2n-1}}(1+o(1))^{2n}+\sum_{j=2}^{2n-1}a_j\left(\eta^{\frac{1}{2n-1}}(1+o(1))\right)^{j}- \eta^{\frac{2n}{2n-1}}(1+o(1))  \nonumber\\
&=&\frac{1}{2n}\eta^{\frac{2n}{2n-1}}(1+o(1))+\sum_{j=2}^{2n-1}a_j\eta^{\frac{j}{2n-1}}(1+o(1))
- \eta^{\frac{2n}{2n-1}}(1+o(1))\nonumber\\
&=&\left(\frac{1-2n}{2n}\right)\eta^{\frac{2n}{2n-1}}(1+o(1))+\sum_{j=2}^{2n-1}a_j\eta^{\frac{j}{2n-1}}(1+o(1))\\
&=&\left(\frac{1-2n}{2n}\right)\eta^{\frac{2n}{2n-1}}\left[(1+o(1))+\left(\frac{2n}{1-2n}\right)\left(\sum_{j=2}^{2n-1}a_j\eta^{\frac{j-2n}{2n-1}}(1+o(1))\right)\right]\\
&=&\left(\frac{1-2n}{2n}\right)\eta^{\frac{2n}{2n-1}}(1+o(1))
\end{eqnarray*}
as $|\eta|\rightarrow\infty$, i.e.,
\begin{equation}\label{asymptotic4}
B_\eta(\lambda(\eta))\sim
\left(\frac{1-2n}{2n}\right)\eta^{\frac{2n}{2n-1}}
\end{equation}
as $|\eta|\rightarrow\infty$.  By our definition of $b^*$, the
result is established.
\end{proof}

We will also need asymptotic estimates for
$b^{(j)}(\lambda(\eta))$:
\begin{lemma}\label{lemma: asy b^j} For $j=2,\ldots,2n$,
\begin{equation}\label{asy for b^j}
b^{(j)}(\lambda(\eta))  \sim \frac{(2n-1)!}{(2n-j)!}
\eta^{\frac{2n-j}{2n-1}}\quad\text{as $|\eta|\rightarrow\infty$.}
\end{equation}
\end{lemma}
\begin{proof}
The proof is similar to that for Lemma \ref{lemma: asy behavior of
b*} and is omitted.
\end{proof}

We close this section with a proposition stating several
properties of $b^*$.
\begin{prop}\label{prop: properties of b^*}
For $b$ as in \eqref{our b}, $b^*(\eta)=\sup_x[\eta x -b(x)]$ is
finite and convex on $\mathbb{R}$.  It is therefore continuous.
\end{prop}
\begin{proof}
We merely sketch the proof since these are known properties of the
Legendre transform.  The finiteness of $b^*$ comes from the fact
that $x \mapsto \eta x -b(x)$ is a non-constant polynomial with
even degree and negative leading coefficient.  The convexity comes
from the fact that $b^*$ is the supremum of a family $\{\,\eta
\mapsto \eta x -b(x):x\in \mathbb{R}\,\}$ of convex functions.
Furthermore, $b^*$ is continuous since every (finite) convex
function is continuous.
\end{proof}

\section{\label{section: local properties of B} Local properties of $\lambda(\eta)$ and $B_\eta$}

The main result of this section describes those points in
$\mathbb{R}$ that can be (global) minimizers of one of the members
of the family of polynomials $\{\,B_{\eta}(\lambda):=-\lambda \eta
+b(\lambda):\eta\in \mathbb{R}\,\}$.

\begin{definition}
For each $\eta \in  \mathbb{R}$, define $\Lambda_\eta$ to be the
set of all points at which the polynomial $B_{\eta}$ attains its
global minimum.  Let $\sigma(\eta)$ be the smallest element of
$\Lambda_{\eta}$ and let $\lambda(\eta)$ be the largest. Let
$\mathcal{C}=\{\,\eta: |\Lambda_\eta|>1\,\}$. Finally, let
$\Lambda=\bigcup_{\eta} \Lambda_\eta$ and
$\lambda[\mathbb{R}]=\{\lambda(\eta):\eta \in \mathbb{R}\}$.
\end{definition}

\begin{theorem} \label{theorem: structure of Lambda} $ \lambda[\mathbb{R}] = \mathbb{R} \setminus \bigcup_{c \in \mathcal{C}}[\sigma(c),
\lambda(c))$.
\end{theorem}

In the case of a {\it convex} polynomial $b$, $b'$ is one-to-one
and hence $B_\eta$ has precisely one critical point for each
$\eta$. Thus in the convex case, $\mathcal{C}=\emptyset$ and $b'$
and $\lambda$ are inverses. These statements are not true in the
non-convex case, though there are partial analogues.

Since all elements of $\Lambda_\eta$ are solutions to
$\eta=b'(\lambda)$, the following is immediate.
\begin{cor}\label{cor: lambda inj}
$\eta \mapsto \lambda(\eta)$ is injective.
\end{cor}
We easily verify several other properties of $\lambda(\cdot)$ and
$b'$.
\begin{lemma}\label{lemma: lambda incr}
If $\lambda_1,\lambda_2 \in \Lambda$ with $\lambda_1 < \lambda_2$,
then $b'(\lambda_1)<b'(\lambda_2)$.
\end{lemma}
\begin{proof}
Since $\lambda_i \in \Lambda$, there exist $\eta_1 \ne \eta_2$
such that $\lambda_1=\lambda(\eta_1)$ and
$\lambda_2=\lambda(\eta_2)$. Since $\eta_i =
b'(\lambda(\eta_i))=b'(\lambda_i)$, we must show that
$\eta_1<\eta_2$.

Suppose, on the contrary, that $\eta_2 < \eta_1$.  Since
$\lambda(\eta_i)$ is a point at which $B_{\eta_i}(\lambda)=-\eta_i
\lambda + b(\lambda)$ attains its global minimum,
$B_{\eta_2}(\lambda_2)<B_{\eta_2}(\lambda_1)$. If $\eta_2 <
\eta_1$,
\begin{eqnarray*}
B_{\eta_1}(\lambda_2)-B_{\eta_1}(\lambda_1)&=&-\eta_1 \lambda_2 +
b(\lambda_2)-(-\eta_1 \lambda_1 +b(\lambda_1))\\
&=&(\lambda_1-\lambda_2)[\eta_2+(\eta_1-\eta_2)] +b(\lambda_2) -
b(\lambda_1)\\
&=&(\lambda_1-\lambda_2)(\eta_1-\eta_2)+B_{\eta_2}(\lambda_2)-B_{\eta_2}(\lambda_1)<0.\\
\end{eqnarray*}
This contradicts the fact that $B_{\eta_1}$ takes its global
minimum at $\lambda_1$ and proves the result.
\end{proof}

\begin{lemma}\label{lemma: lambda and b' are inverses}
$\lambda:\mathbb{R}\to\lambda[\mathbb{R}]$ and
$b':\lambda[\mathbb{R}] \to \mathbb{R}$ are inverses.
\end{lemma}
\begin{proof}
We have already observed that $\eta = b'(\lambda(\eta))$ for all
$\eta \in \mathbb{R}$.

Thus consider $\omega \in \lambda[\mathbb{R}]$.  There exists a
unique $\nu$ such that $\omega = \lambda(\nu)$.  Since $\nu =
b'(\lambda(\nu))=b'(\omega)$, we have $\omega =
\lambda(b'(\omega))$, as desired.
\end{proof}

\begin{cor}\label{cor: lambda incr}
$\lambda:\mathbb{R}\to\lambda[\mathbb{R}]$ is increasing.
\end{cor}

The proof of Theorem \ref{theorem: structure of Lambda} requires a
number of additional technical lemmas.

\begin{lemma}\label{lemma: disjoint safe zones}
Take $c \in \mathcal{C}$.  If
$\omega\in(\sigma(c),\lambda(c))\setminus \Lambda_c$, then there
does not exist an $\eta$ for which $\omega\in \Lambda_\eta$.
\end{lemma}
\begin{proof}
Since $\omega$ is not a location of the global minimum of $B_c$,
$$-\omega c+b(\omega)> -\sigma(c)c+b(\sigma(c)) \quad\text{and}\quad -\omega c+b(\omega)>-\lambda(c) c + b(\lambda(c)) .$$
Since $\sigma(c)< \omega < \lambda(c)$, if $\eta >c$,
\begin{eqnarray*}
B_{\eta}(\omega)-B_{\eta}(\lambda(c))&=&-\eta \omega
+b(\omega)+\eta \lambda(c)
-b(\lambda(c))\\
&=&(-\omega+\lambda(c))[c+(\eta - c)]+b(\omega)-b(\lambda(c))\\
&=&B_c(\omega)-B_c(\lambda(c))+[\lambda(c) -\omega][\eta-c]>0.
\end{eqnarray*}
Similarly, for $\eta<c$, $B_{\eta}(\omega)-B_{\eta}(\sigma(c))>0$.
We conclude that there is no $\eta \in \mathbb{R}$ for which
$\omega$ is the location of the global minimum of $B_{\eta}$.
\end{proof}

\begin{cor}\label{cor: disjoint safe zones}
If $\eta_1 \ne \eta_2$, then $\Lambda_{\eta_1} \cap
\Lambda_{\eta_2}=\emptyset$.  Furthermore, if  $c_1,c_2 \in
\mathcal{C}$ with $c_1\ne c_2$, $[\sigma(c_1),\lambda(c_1))\cap
[\sigma(c_2),\lambda(c_2))=\emptyset$.
\end{cor}

\begin{lemma}\label{lemma: C is finite}
Let $\deg b = 2n$.  Then $|\mathcal{C}| \le n-1$.
\end{lemma}
\begin{proof}
Let $c\in\mathcal{C}$.  Then $(\sigma(c), \lambda(c))$ is
non-empty. Since $\lambda \mapsto -c\lambda+b(\lambda)$ takes the
same value at $\sigma(c)$ and $\lambda(c)$, by Rolle's theorem,
there exists $\omega_0 \in (\sigma(c),\lambda(c))$ at which
$-c+b'(\omega_0)=0$, i.e., $\omega_0$ is another critical point of
$B_c$.  Since $-c+b'(\sigma(c)+)>0$ but $-c+b'(\lambda(c)-)<0$, we
may take the point $\omega_0$ to be a local maximum of $B_c$. Thus
$B''_c=b''$ must change sign in each of $(\sigma(c),\omega_0)$ and
$(\omega_0,\lambda(c))$. Since by Corollary \ref{cor: disjoint
safe zones} the intervals in the collection
$\{\,(\sigma(c),\lambda(c)):c\in\mathcal{C}\,\}$ are disjoint, the
total number $|\mathcal{C}|$ can not exceed $\frac{1}{2}\deg
b''=n-1$.
\end{proof}

The next lemma is central, as it identifies subintervals of
$\mathbb{R}$ in which $\lambda[\mathbb{R}]$ is dense.
\begin{lemma}\label{lemma: non-empty} Let $\alpha,\eta_0
\in\mathbb{R}$.

\begin{enumerate}
\item If $a<\sigma(\eta_0)$, $(a,\sigma(\eta_0)) \cap
\lambda[\mathbb{R}]\ne\emptyset$.

\item If $\lambda(\eta_0)<a$, $(\lambda(\eta_0),a)\cap
\lambda[\mathbb{R}]\ne\emptyset$.
\end{enumerate}
\end{lemma}
\begin{proof}
We note that $\eta_0$ need not be an element of $\mathcal{C}$.  If
it is not, $\sigma(\eta_0)=\lambda(\eta_0)$.

We prove the first statement.  The proof of the second is similar.
If $\omega < \sigma(\eta_0)$, then
\begin{equation}
B_{\eta_0}(\omega)>B_{\eta_0}(\sigma(\eta_0)).
\end{equation}
Fix $a<\sigma(\eta_0)$.  Since $B_{\eta_0}$ is continuous, for any
$L>0$ satisfying $-L<a$, there exists $d$ (depending on $L$) such
that for all $\omega \in [-L,a]$,
$$B_{\eta_0}(\omega)\ge d >B_{\eta_0}(\sigma(\eta_0)) \iff B_{\eta_0}(\omega)-B_{\eta_0}(\sigma(\eta_0))\ge d-B_{\eta_0}(\sigma(\eta_0)):=\alpha >0.$$
We choose $L$ as follows: Since, by Lemma \ref{lemma: asy behavior
of lambda}, $\lambda(\eta) \to -\infty$ as $\eta \to -\infty$,
there exists $\eta^*<\eta_0 -1$ satisfying
$\lambda(\eta^*)<-|a|-1$. Set $-L:=\lambda(\eta^*)$.

Set $\varepsilon = \min\{1, \frac{\alpha}{2\sigma(\eta_0)+L}\}$. We
claim that for all $\eta \in (\eta_0-\varepsilon, \eta_0)$ and
$\omega \in [-L,a]$,
\begin{equation}\label{lower bound on rectangle}
B_\eta(\omega)>B_\eta(\sigma(\eta_0)).
\end{equation}
Indeed, since
$B_\eta(\omega)=B_{\eta_0}(\omega)-(\eta-\eta_0)\omega$ and
$B_\eta(\sigma(\eta_0))=B_{\eta_0}(\sigma(\eta_0))-(\eta -
\eta_0)\sigma(\eta_0)$,
\begin{eqnarray*}
B_\eta(\omega)-B_\eta(\sigma(\eta_0))&=&B_{\eta_0}(\omega)-B_{\eta_0}(\sigma(\eta_0))+(\eta-\eta_0)(\sigma(\eta_0)-\omega)\\
&\ge&\alpha -\varepsilon(\sigma(\eta_0)-\omega)\\
&\ge&\alpha -\frac{\alpha}{2(
\sigma(\eta_0)+L)}(\sigma(\eta_0)+L)=\frac{\alpha}{2}.
\end{eqnarray*}
This proves \eqref{lower bound on rectangle}.

Finally, we claim that if $\eta \in (\eta_0-\varepsilon, \eta_0)$,
then $\lambda(\eta) \in (a, \sigma(\eta_0))$.  Since $\eta <
\eta_0$ and $\lambda$ is an increasing function,
$\lambda(\eta)<\lambda(\eta_0)$. By Lemma \ref{lemma: disjoint
safe zones} this forces $\lambda(\eta)<\sigma(\eta_0)$.  Since
$B_\eta(\lambda(\eta))<B_\eta(\sigma(\eta_0))$, by \eqref{lower
bound on rectangle}, $\lambda(\eta)\notin [-L,a]$. If
$\lambda(\eta)$ were less than $-L=\lambda(\eta^*)$, then the fact
that $\lambda$ is increasing would imply $\eta<\eta^*<\eta_0-1$,
which is false.  We conclude that
$\lambda(\eta)\in(a,\sigma(\eta_0))$.
\end{proof}

We are now ready to prove Theorem \ref{theorem: structure of
Lambda}.
\begin{proof}
That $\lambda[\mathbb{R}] \subseteq \mathbb{R}\setminus
\bigcup_{c\in\mathcal{C}} [\sigma(c),\lambda(c))$ follows from
Lemma \ref{lemma: disjoint safe zones} and Corollary \ref{cor:
disjoint safe zones}.

Next we prove $\mathbb{R} \setminus \bigcup_{c \in
\mathcal{C}}[\sigma(c), \lambda(c)) \subseteq
\lambda[\mathbb{R}]$. Suppose $\mathcal{C}\ne\emptyset$. Since
$|\mathcal{C}|$ is finite, we may order the elements of
$\mathcal{C}$ so that $c_i < c_{i+1}$, $1 \le i \le k-1$.  The
left-hand set is made up of three kinds of intervals: two
semi-infinite intervals $(-\infty, \sigma(c_1))$ and
$[\lambda(c_k), \infty)$, and (if $k \ge 2$) the intervals
$[\lambda(c_i), \sigma(c_{i+1}))$. We must show that every
$\omega$ in one of these intervals is in $\lambda[\mathbb{R}]$.

Consider first the case in which $k \ge 2$ and $\omega \in
(\lambda(c_i), \sigma(c_{i+1}))$.  Set
$$U:=(\lambda(c_i),\omega)\cap \lambda[\mathbb{R}], \quad\quad V:=(\omega, \sigma(c_{i+1}))\cap\lambda[\mathbb{R}], \quad\quad \nu:=\inf\{\,b'(\lambda):\lambda\in V\,\}.$$
By Lemma \ref{lemma: non-empty}, $V \ne \emptyset$ and hence $\nu$
is defined.  We claim $\omega = \lambda(\nu)$.

If $\lambda(\eta) \in V$, then
$\lambda(\eta)<\sigma(c_{i+1})<\lambda(c_{i+1})$, and the
monotonicity of $b'$ on $\lambda[\mathbb{R}]$ implies
$$\nu \le b'(\lambda(\eta))<b'(\lambda(c_{i+1}))=c_{i+1}.$$
Furthermore, since $U$ contains some $\lambda(\eta_0)$, for
$\lambda(\eta)\in V$,
$\lambda(c_i)<\lambda(\eta_0)<\omega<\lambda(\eta)$, so that $c_i
< \nu$.  It follows from the monotonicity of $\lambda(\cdot)$ that
$\lambda(c_i) < \lambda(\nu)< \sigma(c_{i+1})$. Thus either
$\lambda(\nu)\in U$, $\lambda(\nu) \in V$, or
$\lambda(\nu)=\omega$.

Suppose $\lambda(\nu)\in U$. By Lemma \ref{lemma: non-empty},
$(\lambda(\nu),\omega)\cap \lambda[\mathbb{R}]$ is not empty.  It
thus contains $\lambda(\eta_0)$ for some $\eta_0 > \nu$.  But then
$\eta_0$ would be a lower bound for $\{\,b'(\lambda):\lambda \in
V\,\}$, contradicting the definition of $\nu$. Thus $\lambda(\nu)
\notin U$.

Suppose $\lambda(\nu) \in V$. Since $\nu \notin \mathcal{C}$,
$\lambda(\nu)=\sigma(\nu)$.  Consider $(\omega, \sigma(\nu))\cap
\lambda[\mathbb{R}]$.  By Lemma \ref{lemma: non-empty}, this is
not empty, and thus there exists $\lambda(\eta_0)$ in this set,
hence in $V$, with
$\eta_0=b'(\lambda(\eta_0))<b'(\lambda(\nu))=\nu$. This
contradicts the fact that $\nu$ is a lower bound for
$\{\,b'(\lambda):\lambda \in V\,\}$. Thus $\lambda(\nu)\notin V$.
We conclude that $\lambda(\nu)=\omega$.

The proof in the case of the semi-infinite intervals $(-\infty,
\sigma(c_1))$ and $(\lambda(c_k),\infty)$ is virtually identical.
If $|\mathcal{C}|=0$, one can take $\sigma(c_1)(=\lambda(c_1))$
arbitrarily large since $\lambda(\eta)\to\infty$ as $\eta \to
\infty$ to conclude that $\lambda[\mathbb{R}]=\mathbb{R}$.
\end{proof}

This theorem, together with Lemma \ref{lemma: C is finite}, yields
the following:
\begin{cor}
$b$ is convex on $\mathbb{R}\setminus
\bigcup_{c\in\mathcal{C}}[\sigma(c),\lambda(c))$.
\end{cor}

\section{\label{section: estimates for I} Estimates for $\int_{-\infty}^\infty e^{-2\tau p_\eta(\xi)}\,d\xi$}

Recall that $p_\eta(\xi) = -\eta \xi +
b(\xi+\lambda(\eta))-b(\lambda(\eta))$.  In what follows, we will
sometimes suppress the dependence of $p$ and $\lambda$ on $\eta$.
Within this section, we define
\begin{equation}\label{def of I in section 4}
I:=\int_{-\infty}^\infty e^{-2\tau p_\eta(\xi)}\,d\xi.
\end{equation}

In the paper \cite{GH:11}, we obtained sharp estimates on
integrals of the form $I$ that are uniform in the coefficients of
$p$ {\it under the hypothesis} that $p$ has degree four.  We
showed there that those estimates do not generalize to polynomials
of higher degree.  The estimates obtained below are less precise
but are nonetheless sufficient to prove our results on absolute
convergence of the integral defining Szeg\"o kernel.

\subsection{Estimates for $I$ for convex $p$.}
As in the fourth-degree setting, our analysis makes use of known
results on the integral of $e^{-p}$ over intervals on which $p$ is
convex.  Such results do {\it not} require $p$ to have degree
four. We recall the main result here:

\begin{lemma}[Lemma 4.9, \cite{GH:11}] \label{lemma: for convex p} Let $n$ be a positive integer and define
$p(\xi)=\displaystyle{\sum_{j=2}^{2n} \beta_j \xi^j}$. Suppose $p$
is convex on $J$, where $J$ is one of the intervals
$(-\infty,\infty)$, $(0,\infty)$, or $(-\infty,0)$. Then
\begin{equation}\label{est. for convex p}
\int_J e^{-p(\xi)}\,d\xi \approx
\left[\sum_{j=2}^{2n}|\beta_j|^{\frac{1}{j}} \right]^{-1}.
\end{equation}
\end{lemma}
This lemma follows from work of Bruna, Nagel, and Wainger
\cite{BrNagWa:88}.  A more detailed discussion, including
variations and proofs, can be found in Section 4.2 of
\cite{GH:11}.

\subsection{A lower bound for $I$ for non-convex $p$.}
By construction, $p_\eta$ vanishes to at least second order at the
origin.  Thus
$$p_\eta(\xi)=\sum_{j=2}^{2n} \frac{p^{(j)}(0)}{j!}\xi^j=\sum_{j=2}^{2n} \frac{b^{(j)}(\lambda(\eta))}{j!} \xi^j \le \frac{1}{2}\sum_{j=2}^{2n} |b^{(j)}(\lambda(\eta))||\xi|^j, $$
and (suppressing the dependence of $\lambda$ on $\eta$)
\begin{eqnarray}
I &\ge& \int_{-\infty}^\infty e^{-\tau \sum_{j=2}^{2n}
|b^{(j)}(\lambda)||\xi|^j}\,d\xi\nonumber\\
&\ge&\int_{0}^\infty e^{-\sum_{j=2}^{2n}
 \tau |b^{(j)}(\lambda)|\xi^j}\,d\xi\nonumber\\
 &\approx&\left[\sum_{j=2}^{2n}\tau^{\frac{1}{j}}
 |b^{(j)}(\lambda)|^{\frac{1}{j}}\right]^{-1}\label{lower bound on I},
\end{eqnarray}
where in the last line we have used Lemma \ref{lemma: for convex
p} applied to $\xi \mapsto \sum_{j=2}^{2n} \tau |b^{(j)}(\lambda)|
\xi^j$, which is clearly a convex polynomial on $(0,\infty)$.

\subsection{An upper bound for $I$ for non-convex $p$.}
An upper bound for $I$ will give rise to a lower bound for the
factor $[N(\eta,\tau)]^{-1}$ appearing in the integrand for the
Szeg\"o kernel.  These estimates are therefore necessary for the
results on the divergence of the integral $S[(x,y,t),(r,y,t)]$. We
will see in Section \ref{section: proof of thms} that for $M$
sufficiently large, the contribution to $S$ from
$\{(\eta,\tau):|\eta|>M, \tau
>0\}$ is finite for any $x$ and $r$.  Thus when
$S[(x,y,t),(r,y,t)]$ is divergent, it is because there exists some
finite $\eta_0$ for which the contribution to $S$ from
$\{\,(\eta,\tau): \eta_0 < \eta <\eta_0+\varepsilon, \tau >0\,\}$
is infinite.  The following proposition is therefore sufficient to
establish these results.

\begin{prop} \label{prop: upper bound on I}
Fix $\eta_0\in\mathbb{R}$ and $\varepsilon > 0$.  Then there
exists $c:=c(\eta_0,\varepsilon )$ such that for all $\eta \in
(\eta_0,\eta_0+\varepsilon)$ and for all $\tau>0$,
\begin{equation}\label{upper bound for I (local)}
I \le c \frac{1+\tau^\frac{1}{2}}{\tau^\frac{1}{2}}.
\end{equation}
\end{prop}

We begin by factoring $p_\eta$. For fixed $\eta$, this is a
non-negative polynomial vanishing to even order at the origin, its
real roots are of even multiplicity, and its non-real roots occur
in complex conjugate pairs. Its factorization over $\mathbb{C}$
may therefore be written
\begin{equation}\label{factorization over C}
p_\eta(\xi)=\frac{1}{2n}\xi^2\prod_{j=2}^n[\xi-\alpha_j(\eta)][\xi-\overline{\alpha_j(\eta)}],
\end{equation}
where the $\alpha_j$ may be real and need not be distinct.
Furthermore, if $\alpha_j(\eta)=h_j(\eta)+ik_j(\eta)$, we order
the roots so that $h_2(\eta)\le h_3(\eta)\le \ldots \le
h_n(\eta)$. The factorization of $p_\eta$ over $\mathbb{R}$ is
thus
\begin{equation}\label{factorization into quadratics}
p_\eta(\xi)=\frac{1}{2n}\xi^{2}\prod_{j=2}^{n}[(\xi-h_j(\eta))^2+k^2_j(\eta)].
\end{equation}
In what follows, we denote the $j$-th quadratic factor in the
above product by $q_j(\xi,\eta)$.

Since the $h_j$ are functions of $\eta$ and we seek estimates for
$I$ that are valid for all $\eta$ throughout an interval, we need
a lemma on the local behavior of the $h_j$:

\begin{lemma}\label{lemma: local bddness of h_j}
Fix $\eta_0 \in \mathbb{R}$ and $\varepsilon >0$.  Then there
exists $C>1$ such that for all $\eta \in J:=
[\eta_0,\eta_0+\varepsilon]$ and for all $j$, $|h_j(\eta)|\le
C-1$.
\end{lemma}
\begin{proof}
This is a standard argument. Suppose the result fails, so that for
some $j$, $h_j$ is unbounded on $J$. Assume without loss of
generality that $j=2$. Let $(\eta_\ell)$ be a sequence in $J$ for
which $|h_2(\eta_\ell)| \to \infty$.  By extracting subsequences
if necessary, we may assume that $(\eta_\ell)$ converges to some
$\eta' \in J$.

Recall that, by definition, $p_\eta(\xi)=-\eta(\xi +
\lambda(\eta))+b(\xi+\lambda(\eta))-b^*(\eta)$.  Since $b,b^*$ are
continuous and $\lambda$ is bounded on $J$, for $\xi$ fixed, there
exists $M_\xi >0$ such that for all $\eta \in J$, $0 \le
p_\eta(\xi)\le M_\xi$.  Fix $\xi=1$.  Then for all $\ell$,
$$0\le \frac{1}{2n}\prod_{j=2}^n q_j(1,\eta_\ell)\le M_1 .$$
Since $\lim_{\ell\to\infty}q_2(1,\eta_\ell)=\infty$, $\lim_{\ell
\to \infty} \prod_{j=3}^n q_j(1,\eta_\ell)=0$.  Thus there exists
a factor $q_{j_1}$ and a subsequence $(\eta^{(1)}_\ell)$ such that
$\lim_{\ell \to \infty}q_{j_1}(1,\eta^{(1)}_\ell)=0$. This forces
$\lim_{\ell \to \infty} h_{j_1}(\eta^{(1)}_\ell)=1$.

Now take $\xi=2$. It is still the case that $\lim_{\ell to \infty
}q_2(2,\eta^{(1)}_\ell)=\infty$, but now for $\ell$ sufficiently
large, $q_{j_1}(2,\eta^{(1)}_\ell)$ is bounded away from zero.
These facts, together with the  boundedness of $p_\eta(2)$ on $J$,
allow us to find a different factor $q_{j_2}$ and a subsequence
$(\eta^{(2)}_\ell)$ of $(\eta^{(1)}_\ell)$ such that
$q_{j_2}(2,\eta^{(2)}_\ell) \to 0$ and $h_{j_2}(\eta^{(2)}_\ell)
\to 2$.  Repeating this process at most $n-2$ times leads to a
subsequence $(\nu_\ell)$ of the original such that
$h_{j_i}(\nu_\ell)$ tends to $i$.

Fix $\xi=n$.  There exists $M_n$ such that $0\le p_\eta(n)\le M_n$
for all $\eta \in J$. Furthermore, each $q_{j}(n, \nu_\ell)$,
$2\le j \le n$ is bounded away from zero for $\ell$ sufficiently
large, but $q_1(n,\nu_\ell)$ is unbounded.  This is a
contradiction, and the lemma is proved.
\end{proof}

We now prove Proposition \ref{prop: upper bound on I}. With
notation as in the proof of Lemma \ref{lemma: local bddness of
h_j}, we find
\begin{eqnarray*}
I&=&\int_{-\infty}^\infty \exp\left(-\frac{\tau}{n}\xi^2
\prod_{j=2}^n [(\xi-h_j)^2+k^2_j]\right) \,d\xi \\
&\le&\int_{-\infty}^\infty \exp\left(-\frac{\tau}{n}\xi^2
\prod_{j=2}^n (\xi-h_j)^2\right) \,d\xi.
\end{eqnarray*}
Write this last integral as the sum of integrals $I_1$, $I_2$, and
$I_3$, where $I_1$ is over the interval $(-\infty ,C)$, $I_2$ is
over $[-C, C]$, and $I_3$ is over  $(C,\infty)$.

Since $-(C-1)\le h_j(\eta) \le C-1$ for each $j$ and for all $\eta
\in J$,
\begin{eqnarray*}
I_1 &\le&\int_{-\infty}^{-C} \exp\left(-\frac{\tau}{n}\xi^2
\prod_{j=2}^n (-C-h_j)^2\right) \,d\xi\\
&\le&\int_{-\infty}^\infty \exp\left(-\frac{\tau}{n}\xi^2
\prod_{j=2}^n (-C-h_j)^2\right) \,d\xi\\
&\approx& \left(\frac{\tau}{n}\prod_{j=2}^n
(C+h_j)^2\right)^{-\frac{1}{2}}\\
&\approx& \tau^{-\frac{1}{2}}
\end{eqnarray*}
since $2C-1 \ge C+h_j(\eta) \ge 1$ for all $j$ and for all $\eta
\in J$.

We make the simplest possible estimate of $I_2$; since the
integrand is less than 1,
$$I_2 \le 2C\approx 1. $$

We estimate $I_3$ in the same way as $I_1$, using now the fact
that for all $j$, $h_j(\eta) < C -1$ for all $\eta\in J$.
\begin{eqnarray*}
I_3&\le&\int_C^\infty \exp\left(-\frac{\tau}{n}\xi^2
\prod_{j=2}^n (C-h_j)^2\right) \,d\xi\\
&\le&\int_{-\infty}^\infty \exp\left(-\frac{\tau}{n}\xi^2
\prod_{j=2}^n (C-h_j)^2\right) \,d\xi\\
&\approx& \left(\frac{\tau}{n}\prod_{j=2}^n
(C-h_j)^2\right)^{-\frac{1}{2}}\\
&\approx& \tau^{-\frac{1}{2}}.
\end{eqnarray*}
Putting the three estimate together yields $I \lesssim
\max\{\tau^{-\frac{1}{2}}, 1 \} \approx
\frac{1+\tau^{\frac{1}{2}}}{\tau^\frac{1}{2}}$, as claimed.

\section{\label{section: proof of thms} Proofs of Theorems}

If we show that for all non-negative integers $i_1$,$j_1$, $i_2$
and $j_2$, each integral
\[
\int\!\!\!\int_{\tau>0} e^{\eta \tau[z_1+\bar{w}_1] +i\tau[z_2 -
\bar{w}_2] }\frac{\eta^{i_1+j_1}
\tau^{i_1+j_1+i_2+j_2+1}}{N(\eta,\tau)} \,d\eta\,d\tau
\]
is absolutely convergent in the region in which
\[
h+k+b(x)+b(r)-2b^{**}\left( \frac{x+r}{2}\right)>0,
\]
it will follow that this integral is in fact equal to
$\partial^{i_1}_{z_1}
\partial^{j_1}_{\bar{w}_1}\partial^{i_2}_{z_2}
\partial^{j_2}_{\bar{w}_2}S(z,w)$.

Set $\delta=h+k$, $(z_1,z_2)=(x+iy,t+ib(x)+ih)$,
$(w_1,w_2)=(r+is,u+ib(r)+ik)$, $s=i_1+j_1$, and
$m=i_1+j_1+i_2+j_2$ (so that $m \ge s$). The integral becomes
\begin{equation}\label{S(s,m,delta)}
S^{s,m,\delta}:=\int \!\!\! \int_{\tau>0} e^{\eta \tau
[x+r+i(y-s)]+i\tau[t-u+i(b(x)+b(r)+\delta)]}\frac{\eta^{s}
\tau^{m+1}}{N(\eta,\tau)}\,d\eta\,d\tau,
\end{equation}
and it converges absolutely if and only if
\begin{equation}\label{S tilde}
\widetilde{S}^{s,m,\delta} := \int_{-\infty}^\infty
\!\int_{0}^\infty e^{-\tau[\delta +b(x)+b(r)-\eta(x+r)]}
\frac{|\eta|^s \tau^{m+1}}{N(\eta,\tau)} \,d\tau\,d\eta < \infty.
\end{equation}
From \eqref{lower bound on I},
\begin{eqnarray*}
\lefteqn{\widetilde{S}^{s,m,\delta}}& &\\
&=&\int_{-\infty}^\infty \!\int_{0}^\infty e^{-\tau[\delta
+b(x)+b(r)-\eta(x+r)]}
\frac{|\eta|^s \tau^{m+1}}{e^{2\tau b^*(\eta)}I} \,d\tau\,d\eta\\
&\ge&\sum_{j=2}^{2n}\int_{-\infty}^\infty \!\int_{0}^\infty
e^{-\tau[\delta +b(x)+b(r)-\eta(x+r)+2b^*(\eta)]} |\eta|^s
\tau^{m+1+\frac{1}{j}}|b^{(j)}[\lambda(\eta)]|^{\frac{1}{j}}
\,d\tau\,d\eta\\
&:=&\sum_{j=2}^{2n}\int_{-\infty}^\infty
I_j^{s,m,\delta}(\eta)\,d\eta\\
&:=&\sum_{j=2}^{2n} I_j^{s,m,\delta}.
\end{eqnarray*}
Further, set $A(x,r,\eta):=b(x)+b(r)-\eta(x+r)+2b^*(\eta)$.  If
$\delta+A(x,r,\eta)>0$, we can evaluate the $\tau$ integral.
\begin{equation}\label{I_j(eta)}
I^{s,m,\delta}_j(\eta) \approx \frac{|\eta|^s
|b^{(j)}(\lambda)|^{\frac{1}{j}}}{[\delta +
A(x,r,\eta)]^{m+2+\frac{1}{j}}}.
\end{equation}
We see that there are two possible barriers to the convergence of
the integrals $I^{s,m,\delta}_j$:
\begin{enumerate}
\item insufficient decay of $I_j^{s,m,\delta}(\eta)$ for fixed
$x$, $r$ as $|\eta| \to \infty$;

\item vanishing of $\delta + A(x,r,\eta)$ at some finite $\eta$
for certain choices of $x$, $r$, and $\delta$.
\end{enumerate}
We deal with these in turn.

\subsection{Behavior of $I_j^{s,m,\delta}(\eta)$ for large $|\eta|$.}

\begin{lemma}\label{lemma: delta+A}
 Fix $x,r\in\mathbb{R}$, $\delta >0$.  Then
$$\delta + A(x,r,\eta) \sim  \left(\frac{2n-1}{n}\right)\eta^{\frac{2n}{2n-1}}, \quad |\eta| \to \infty.$$
\end{lemma}
\begin{proof} This follows immediately from Lemma \ref{lemma: asy behavior of
b*}, since
\begin{eqnarray*}
\delta + A(x,r,\eta)&=&\delta + b(x)+b(r)-\eta(x+r)+2b^*(\eta)\\
&=& \left(\frac{2n-1}{n}\right)\eta^{\frac{2n}{2n-1}}(1+o(1))
\end{eqnarray*}
as $|\eta| \to \infty$.
\end{proof}

We may use Lemmas \ref{lemma: asy b^j} and \ref{lemma: delta+A} to
estimate the integrals for large $|\eta|$:
\begin{eqnarray*}
|I_j^{s,m,\delta}(\eta)| &\sim& c \frac{|\eta|^s }{\eta^{\frac{2n(m+2)}{2n-1}}}\frac{|\eta|^{\frac{2n-j}{j(2n-1)}}}{\eta^{\frac{2n}{j(2n-1)}}}\\
    &=&c |\eta|^{-2-(m-s)-\frac{m+3}{2n-1}}
\end{eqnarray*}
as $|\eta|\rightarrow \infty$. Since $m \ge s \ge 0$,
$-2-(m-s)-\frac{m+3}{2n-1}<-2$. Thus for any fixed $s$, $m$, $j$,
and $\delta>0$, $I^{s, m, \delta}_j$ is convergent at infinity.

\subsection{Vanishing of $\delta + A(x,r,\eta)$} In light of the
previous section, we see that whether or not one of the integrals
$I_j$ converges depends on whether or not the function $\eta
\mapsto \delta+A(x,r,\eta)$ vanishes at a finite $\eta_0$ for some
fixed $x$, $r$, $\delta$, and the behavior of this function near
such zeros.  In fact, we have proved
\begin{prop} If for some fixed $x$, $r$, and $\delta$
\begin{equation}\label{inf strictly pos}
\inf_\eta \delta+A(x,r,\eta)>0,
\end{equation}
then all the $I^{s,m,\delta}_j$ are finite.
\end{prop}
Furthermore,
\begin{eqnarray*}
\inf_\eta \delta+A(x,r,\eta)&=&\delta
+b(x)+b(r)-2\sup\left[\eta\left(\frac{x+r}{2}\right)-b^*(\eta)\right]\\
&=&\delta+b(x)+b(r)-2 b^{**}\left(\frac{x+r}{2}\right),
\end{eqnarray*}
where the convexity of $b^{*}$ and its super-linear growth at
infinity (Lemma \ref{lemma: asy behavior of b*}) guarantee the
finiteness of the supremum in the first line.  It follows that the
integrals defining the Szeg\"o kernel and all of its derivatives
converge absolutely in the region in which
$$\delta+b(x)+b(r)-2b^{**}\left(\frac{x+r}{2} \right)>0.$$ This is precisely the region defined in \eqref{region for abs conv}.  To prove
the remainder of Theorem \ref{theorem: absolute convergence}, we
must use the results of Section \ref{section: local properties of
B} to identify points
$(z,w)=[(z_1,z_2),(w_1,w_2)]=[(x+iy,t+i(b(x)+h)),(r+is,u+i(b(r)+k))]$
satisfying \eqref{region for abs conv}.

If $(z,w)\in (\Omega \times \overline{\Omega})\cup
(\overline{\Omega} \times \Omega)$, $\delta >0$.  Since
$A(x,r,\eta)\ge 0$, such $(z,w)$ are indeed in the region
\eqref{region for abs conv}.  We turn our attention, then, to
points $(z,w)\in \partial \Omega \times \partial \Omega$, where
$\delta = 0$.

Set
\begin{equation}\label{Ax and Ar}
A_x(\eta):=b^*(\eta)-[\eta x -b(x)] \quad\quad
A_r(\eta):=b^*(\eta)-[\eta r - b(r)],
\end{equation}
so that
\begin{equation}\label{A}
A(x,r,\eta)=A_x(\eta)+ A_r(\eta).
\end{equation}
Fix $x$ and $r$, and recall the definition of $\Lambda_{\eta_0}$
from \eqref{Lambda}. $A(x,r,\eta_0)=0$ if and only if
$A_x(\eta_0)=A_r(\eta_0)=0$. This, in turn, happens precisely when
$x,r\in\Lambda_{\eta_0}$.

From Lemma \ref{lemma: asy behavior of b*} and Proposition
\ref{prop: properties of b^*}, it follows that $A(x,r,\cdot)$ is a
continuous function of $\eta$ which grows at infinity like
$c|\eta|^{\frac{2n}{2n-1}}$.  Thus if for some fixed $x$ and $r$
it does not vanish, it is bounded below by a positive constant.
Together with Lemma \ref{lemma: delta+A} this shows that if
$(z,w)\in (\partial \Omega \times \partial \Omega) \setminus
\Sigma$, the integrals defining the Szeg\"o kernel and all its
derivatives are absolutely convergent.

Finally, we turn to the proof of Theorem \ref{theorem: x=-r
singularity}.  We must consider the integrals $S^{0,0,\delta}$ and
$\widetilde{S}^{0,0,\delta}$ from \eqref{S(s,m,delta)} and
\eqref{S tilde}.  To simplify notation, we drop the additional
superscripts. We must show
\begin{enumerate} \item[(i)]  $\widetilde{S}^0$ is
divergent, and

\item[(ii)] $\lim_{\delta \to 0^+} \widetilde{S}^{\delta}=\infty$,
\end{enumerate}
whenever there exists $\eta_0$ such that $x,r\in\Lambda_{\eta_0}$.
Clearly (i) implies (ii) since the integrand of
$\widetilde{S}^\delta$ is non-negative and converges pointwise and
monotonically to the integrand of $\widetilde{S}^0$ as $\delta \to
0^+.$ We thus consider (i).

Fix $\eta_0\in\mathbb{R}$, $\varepsilon>0$, and $x,r\in
\Lambda_{\eta_0}$. Applying Proposition \ref{prop: upper bound on
I},
\begin{equation}
\frac{\tau^{\frac{1}{2}}}{1 +\tau^{\frac{1}{2}}} \lesssim
e^{2\tau b^*(\eta)}N(\eta,\tau)^{-1},
\end{equation}
for all $\tau>0$ and $\eta\in(\eta_0,\eta_0+\varepsilon)$.

Substituting into \eqref{S tilde} and recalling the definition of
$A(=A(x,r,\eta))$ from \eqref{A} gives
\begin{eqnarray}
\widetilde{S}^0&>&\int_{\eta_0}^{\eta_0+\varepsilon}\!\!\!\int_0^\infty \tau e^{-\tau A}e^{2\tau b^*(\eta)} N(\eta,\tau)^{-1}\,d\tau\,d\eta\nonumber\\
&\gtrsim&\int_{\eta_0}^{\eta_0+\varepsilon}\!\!\!\int_0^\infty  \frac{\tau^{\frac{3}{2}} e^{-\tau A}}{1 +\tau^{\frac{1}{2}} }\,d\tau\,d\eta\nonumber\\
&=&\int_{\eta_0}^{\eta_0+\varepsilon}\!\!\!\int_0^\infty \frac{
e^{-\tau} \tau^{\frac{3}{2}}
}{A^2[A^{\frac{1}{2}}+\tau^{\frac{1}{2}}]}\, d\tau\,
d\eta.\label{lower bound on S-tilde-0}
\end{eqnarray}

It is now clear that we need a lemma on the order of vanishing of
$A(x,r,\eta)$ at $\eta_0$.
\begin{lemma}\label{lemma: vanishing of A}
Take $b$ as in \eqref{our b}, $\eta_0\in\mathbb{R}$, and $x\in
\Lambda_{\eta_0}$. Then
\begin{equation*}
A_x(\eta)=(\eta-\eta_0)F_x(\eta)
\end{equation*}
for all $\eta \in (\eta_0,\infty)$, where $F_x$ is bounded on each
interval $(\eta_0,\eta_0+\varepsilon)$.
\end{lemma}
\begin{proof}
Since $\lambda(\eta_0)$ is the largest element of
$\Lambda_{\eta_0}$, $\lambda(\eta_0)\ge x$. Since $\lambda(\cdot)$
injective and increasing, for any $\eta > \eta_0$,
$\lambda(\eta)>x$.  Thus for $\eta > \eta_0$,
\begin{eqnarray*}
A_x(\eta)&=&b(x)-\eta x+b^*(\eta)\\
&=& b(x)-\eta x+\eta \lambda(\eta)-b(\lambda(\eta))\nonumber\\
&=&\eta(\lambda(\eta) - x)+b(x)-b(\lambda(\eta))-\eta_0(\lambda(\eta)-x)+\eta_0(\lambda(\eta)-x)\nonumber\\
&=&(\eta-\eta_0)(\lambda(\eta) - x)-(\lambda(\eta) -x)\left[\frac{b(\lambda(\eta))-b(x)-\eta_0(\lambda(\eta)-x)}{\lambda(\eta)-x}\right]\nonumber\\
&=&(\eta-\eta_0)(\lambda(\eta) - x)-(\lambda(\eta)
-x)\phi_{x}(\eta).\nonumber
\end{eqnarray*}
Observe,
$$\phi_{x}(\eta)
=\frac{b(\lambda(\eta))-\eta_0 \lambda(\eta) - \left[b(x)-\eta_0
x\right]}{\lambda(\eta)-x}=\frac{B_{\eta_0}(\lambda(\eta))
 - B_{\eta_0}(x)}{\lambda(\eta)-x}.$$
Since $x\in \Lambda_{\eta_0}$, the minimality of $B_{\eta_0}(x)$
yields $B_{\eta_0}(\lambda(\eta))
 \ge B_{\eta_0}(x)$ for all $\eta\in\mathbb{R}$.  Therefore $\phi_{x}$ is non-negative
on the interval $(\eta_0,\infty)$.  It follows that for $\eta
> \eta_0$,
\begin{equation}\label{bounded}
A_{x}(\eta)=(\eta-\eta_0)(\lambda(\eta) - x)-(\lambda(\eta) -x)
\phi_{x}(\eta)\ge 0 \iff 1\ge \frac{\phi_x(\eta)}{\eta-\eta_0}.
\end{equation}
Hence, on $(\eta_0,\infty)$,
\begin{eqnarray*}
A_x(\eta)&=&(\eta-\eta_0)(\lambda(\eta) - x)-(\lambda(\eta) -x)\phi_{x}\\
&=&(\eta-\eta_0)(\lambda(\eta) - x)\left[1-\frac{\phi_{x}(\eta)}{\eta-\eta_0}\right]\\
&:=&(\eta-\eta_0)F_x(\eta).
\end{eqnarray*}
By inequality \eqref{bounded} and the local boundedness of
$\lambda(\eta)$, $F_x$ is bounded on each interval
$(\eta_0,\eta_0+\varepsilon)$. This proves the lemma.
\end{proof}

We use this lemma to substitute for $A$ in \eqref{lower bound on
S-tilde-0}
\begin{eqnarray*}
\widetilde{S}^0&\ge&
\int_{\eta_0}^{\eta_0+\varepsilon}\!\!\!\int_0^1
\frac{ e^{-\tau} \frac{\tau^{\frac{3}{2}}}{(\eta-\eta_0)^{2}(F_x(\eta)+F_r(\eta))^2} }{(\eta-\eta_0)^{\frac{1}{2}}(F_x(\eta)+F_r(\eta))^{\frac{1}{2}}+1}\, d\tau\, d\eta\\
&=& \left(\int_0^1\tau^{\frac{3}{2}} e^{-\tau}\,
d\tau\right)\int_{\eta_0}^{\eta_0+\varepsilon}
\frac{ \frac{1}{(\eta-\eta_0)^2(F_x(\eta)+F_r(\eta))^2} }{(\eta-\eta_0)^{\frac{1}{2}}(F_x(\eta)+F_r(\eta))^{\frac{1}{2}}+1}\ d\eta\\
&\approx&\int_{\eta_0}^{\eta_0+\varepsilon}\frac{G(\eta)}{(\eta-\eta_0)^2(F_x(\eta)+F_r(\eta))^2
} \,d\eta,
\end{eqnarray*}
where $G(\eta)$ is right-continuous and bounded away from zero.
Since $F_x+F_r$ is also a locally-bounded positive function, the
divergence of the integral follows.  This completes the proof of
Theorem \ref{theorem: x=-r singularity}.

\bibliographystyle{alpha}
\bibliography{references}

\end{document}